\documentclass[preprint]{imsart}
\RequirePackage[OT1]{fontenc}
\RequirePackage{amsthm,amsmath}
\RequirePackage[numbers]{natbib}
\RequirePackage[colorlinks,citecolor=blue,urlcolor=blue]{hyperref}
\usepackage{amsfonts}
\usepackage{amssymb}
\usepackage[latin1]{inputenc}

\startlocaldefs
\theoremstyle{plain}

\newtheorem{prop}{Proposition}[]
\newtheorem{coro}{Corollary}[]

\theoremstyle{remark}
\newtheorem{rema}{Remark}[]

\theoremstyle{definition}

\newtheorem{defn}{Definition}[]
\endlocaldefs

\begin{document}

\begin{frontmatter}
\title{\Large Marginals of multivariate Gibbs distributions with applications in Bayesian species sampling\thanksref{}}
\runtitle{Marginals of multivariate Gibbs distributions}

\begin{aug}
\author{\fnms{\rm Annalisa} \snm{\rm Cerquetti}\thanksref{t1}\ead[label=e1]{annalisa.cerquetti@gmail.com}}
\ead[label=u1,url]{http://www.foo.com}

\thankstext{t1}{Partially supported by grant PRIN MIUR:2008CEFF37 }
\runauthor{A. Cerquetti}

\affiliation{Sapienza University of Rome}




\address{Dipartimento di Metodi e Modelli per l'Economia, il Territorio e la Finanza \\Università degli Studi di Roma La Sapienza\\Via del Castro Laurenziano, 9
00161 Rome, Italy\\
\printead{e1}\\
\phantom{E-mail: annalisa.cerquetti@gmail.com\ }
}

\address{\rm  \normalsize August 8, 2012}

\end{aug}


\begin{abstract}
Gibbs partition models are the largest class of infinite exchangeable partitions of the positive integers generalizing the product form of the probability function of the two-parameter Poisson-Dirichlet family. Recently those models have been investigated in a Bayesian nonparametric approach to species sampling problems as alternatives to the Dirichlet and the Pitman-Yor process priors.  Here we derive marginals of conditional and unconditional multivariate distributions arising from exchangeable Gibbs partitions to obtain explicit formulas for joint falling factorial moments of corresponding conditional and unconditional Gibbs sampling formulas. Our proofs rely on a known result on factorial moments of sum of non independent indicators. We provide an application to a Bayesian nonparametric estimation of the predictive probability to  observe a species already observed a certain number of times.
\end{abstract}



\end{frontmatter}

\section{Introduction}

Exchangeable {\it Gibbs} partitions  (\citep{gnepit06}) are the largest class of infinite exchangeable partitions of the positive integers generalizing the product form of exchangeable partition probability function (EPPF) of the two parameter $(\alpha, \theta)$ Poisson-Dirichlet partition model (\cite{pit95}, \cite{pityor97}), namely
\begin{equation}
\label{2par}
p_{\alpha, \theta}(n_1, \dots, n_k)= \frac{(\theta +\alpha)_{k-1 \uparrow \alpha}}{(\theta +1)_{n-1}} \prod_{j=1}^{k} (1 -\alpha)_{n_j-1}, 
\end{equation}
for $\alpha \in (0,1)$, $\theta > -\alpha$, $(n_1, \dots, n_k)$ a composition of $n$, $1 \leq k \leq n$ and $(x)_{y \uparrow \alpha}=x(x+\alpha)\cdots(x+(y-1)\alpha)$ generalized rising factorials. Their EPPF is characterized by the {\it Gibbs} product form 
\begin{equation}
\label{EPPFgibbs}
p_{\alpha, V}(n_1, \dots, n_k)= V_{n,k} \prod_{j=1}^{k} (1 -\alpha)_{n_j-1},
\end{equation}
for $\alpha \in (-\infty, 1)$, and $V=(V_{n,k})$ weights satisfying the backward recursive relation $V_{n,k}=(n -k\alpha)V_{n+1, k} + V_{n+1, k+1}$, for $V_{1,1}=1$.  By Theorem 12 in \cite{gnepit06} each element of (\ref{EPPFgibbs}) arises as a probability mixture of extreme partitions, namely: Fisher's (1943) partitions (\cite{fis43}) for $\alpha < 0$, Ewens $(\theta)$ partitions (\cite{ewe72}, \cite{kin75}) for $\alpha=0$, and Poisson-Kingman conditional partitions driven by the stable subordinator (\cite{pit03}) for $\alpha \in (0,1)$.  

By an application of  Eq. (2.6) in \cite{pit06}, given an infinite EPPF in the form (\ref{EPPFgibbs}), for each $n \geq 1$  the corresponding joint distribution of the random vector $(N_{1,n} \dots, N_{K_n, n}, K_n)$ of the sizes and number of the blocks in {\it size biased order} (i.e. in order of their least elements) is given by
\begin{equation}
\label{sizebias2}
\mathbb{P}_{\alpha, V} (N_{1,n}=n_1, \dots, N_{K_n, n}=n_k, K_n=k)= 
\end{equation}
$$=\frac{n!}{n_k(n_k+n_{k-1})\cdots (n_k +\dots+n_1) \prod_{j=1}^k (n_j-1)!} V_{n,k} \prod_{j=1}^k (1 -\alpha)_{n_j -1},
$$
where the combinatorial factor accounts for the number of partitions of 
 $[n]$ in which the $j$-th block in order of appearance has $n_j$ elements. When the order of the blocks is irrelevant an alternative, more tractable coding for the joint distribution (\ref{sizebias2}) is in {\it exchangeable random order} (cfr. Eq. (2.7) in \cite{pit06}) 
\begin{equation}
\label{prior}
\mathbb{P}_{\alpha, V}(N_1^{ex}=n_1, \dots, N_{K_n}^{ex}=n_k, K_n=k)= \frac{n!}{\prod_{j=1}^k n_j!}\frac{1}{k!} V_{n,k} \prod_{j=1}^{k} (1 -\alpha)_{n_j-1},
\end{equation} 
that, from now on, we term {\it multivariate Gibbs distribution} of parameters $(n, \alpha, V)$. Corresponding {\it Gibbs sampling formula}, encoding the partition of $n$ by the vector 
of the numbers of blocks of different sizes,  is obtained by the obvious change of variable in (\ref{prior}) and is given by
\begin{equation}
\label{sampling}
\mathbb{P}_{\alpha, V}(C_{1,n}=c_1, \dots, C_{n,n}=c_n)={n!}V_{n,k}\prod_{i=1}^n \frac{[(1-\alpha)_{i-1}]^{c_i}}{(i!)^{c_i} c_i!},
\end{equation}
for $c_i=\sum_{j=1}^k 1\{n_j=i\}$,  for $i=1, \dots, n$, $\sum_{i=1}^n i c_i=n$ and $\sum_{i=1}^n c_i=k$.
Note that this is the general {\it Gibbs} analog of the {\it Ewens sampling formula} (cfr. \cite{ewe72}) 
\begin{equation}
\label{ewensEPPF}
\mathbb{P}_{\theta}(C_{1,n}=c_1,\dots, C_{n,n}=c_n)= \frac{n! \theta^k}{(\theta)_n} \prod_{i=1}^n \frac{1}{(i)^{c_i} c_i!},
\end{equation}
encoding by the vector of counts the Dirichlet $(\theta)$ partition model, (\cite{fer73, kin75}), whose EPPF is well-known to arise for $\alpha=0$ in (\ref{2par}).
A comprehensive reference for the study of (\ref{sampling}), also called {\it component frequency spectrum}, for general combinatorial random structures is \cite{arr03}.\\

In this paper we study marginals of (\ref{prior}), both conditional and unconditional, in order to derive joint falling factorial moments of corresponding conditional and unconditional sampling formulas. 
Our main motivation comes from applications in Bayesian nonparametric estimation in species sampling problems. In this setting, given $n$ observations from a population of species with multiplicities of the first $k$ species observed $(n_1, \dots, n_k)$, interest may lie in conditional predictive estimation of quantities related to a further sample of $m$ observations (cfr. e.g. \cite{lmp07}, \cite{lpw08}), or in conditional estimation of some diversity index of the whole population, (see e.g. \cite{cer12}). A common {\it prior} assumption in the Bayesian nonparametric approach is that the unknown relative abundances $(P_i)_{i\geq 1}$ of the species in the population  follow a random discrete distribution belonging to the Gibbs family, i.e. are such that, by Kingman's correspondence (cfr. \cite{kin78}),
\begin{equation}
\label{kingm}
 \sum_{(i_1, \dots, i_k)} \mathbb{E} \left[  \prod_{j=1}^k P_{i_j}^{n_j} \right]= V_{n,k} \prod_{j=1}^k (1-\alpha)_{n_j-1},
\end{equation}
where $(i_1, \dots, i_k)$ ranges over all ordered $k$-tuples of distinct positive integers. This is equivalent to assume that the theoretically infinite sequence of species {\it labels} $(X_i)_{i \geq 1}$ is exchangeable with almost surely discrete {\it de Finetti} measure representable as $P(\cdot)=\sum_{i=1}^{\infty} P_i \delta_{Y_i}(\cdot)$, for $(P_i)$ any rearrangement of the ranked frequencies $(P_i^{\downarrow})$ satisfying (\ref{kingm}), independent of $(Y_i) \sim$ IID $H(\cdot)$, for $H$ some non atomic probability distribution. 

Actually the study of {\it conditional Gibbs structures} in this perspective has been initiated  in \cite{lmp07} and \cite{lpw08} and some results for conditional falling factorial moments of {\it components} of (\ref{sampling}) are in \cite {flp12a}.  Nevertheless in those papers some confusion arises  between {\it conditional EPPFs}, and {\it conditional multivariate distributions} of the vector of sizes and number of the blocks in exchangeable random order, which heavily affects the complexity of the proofs. 

Here,  after deriving  marginals of conditional and unconditional {\it multivariate Gibbs distributions}, we obtain {\it joint} falling factorial moments of any order of (\ref{sampling}), both conditional and unconditional, and explicit formulas for some distributions of interest generalizing some particular cases obtained in \cite{flp12a}, in a direct way. Our analysis, besides providing a more effective technique for the study of Gibbs sampling formulas, with a view toward Bayesian nonparametric applications, establishes the first systematic study of joint multivariate distributions arising from Gnedin-Pitman's Gibbs partition models.  The paper is organized as follows: in Section 2 we provide marginals of (\ref{prior}) and, resorting to a result in \cite{johkot05} for sum of non independent indicators, derive general formulas for joint falling factorial moments of (\ref{sampling}), together with some explicit marginal distributions and their expected values. In Section 3 we derive {\it conditional multivariate Gibbs distributions} and their marginals,  for sizes and number of {\it new} blocks induced by the additional $m$-sample. A complete analysis is performed for {\it conditional Gibbs sampling formulas} exploiting the same technique adopted in Section 2. In Section 4 we focus on  {\it multivariate P\'olya-like} distributions arising by the conditional allocation of the additional sample in {\it old} blocks. Finally, in Section 5, we provide an application of marginals of multivariate Gibbs distributions to a Bayesian nonparametric estimation of a  $m$-step ahead probability to detect at observation $n+m+1$ a species already observed a certain number of times. 

\section{Marginals of multivariate Gibbs distributions}
To obtain the marginal distributions for general Multivariate Gibbs distributions (\ref{prior}) it is enough to resort to the definition of generalized {\it central} Stirling numbers (cfr. Eq. 1.9 and 1.19 in  \cite{pit06}) (see the Appendix for further details) 
$$ S_{n,k}^{-1, -\alpha}= \frac{n!}{  k!}\sum_{(n_1, \dots, n_k)} \prod_{j=1}^k \frac{(1-\alpha)_{n_j-1}}{n_j!},
$$
where the sum ranges over all  $(n_1, \dots, n_k)$ compositions of $n$. From now on we refer to (\ref{prior}) omitting the {\it ex} power in the notation. 
\begin{prop} Under a general Gibbs partition model (\ref{EPPFgibbs}) of parameters $(\alpha, V)$, for each $n \geq 1$ the $r$-dimensional marginal of (\ref{prior}), for $0 \leq k-r \leq n -\sum_{j=1}^r n_j$, is given by 
\begin{equation}
\label{rmarg}
\mathbb{P} (N_1=n_1, \dots, N_r=n_r, K_n=k)= 
\end{equation}
$$
=\frac{n!}{\prod_{j=1}^r n_j!}\prod_{j=1}^r (1 -\alpha)_{n_j-1} 
 \frac{V_{n,k}}{k!} \sum_{(b_1, \dots, b_{k-r})} \frac{1}{\prod_i b_i!}\prod_{i=1}^{k-r} (1-\alpha)_{b_i-1}
$$
for $(b_1, \dots, b_{k-r})$ such that  $b_i >0$ $\forall i$ and $\sum_{i} b_i= n-\sum_{j=1}^r n_j$. Multiplying and dividing by $(n-\sum_{j=1}^{r} n_j)!$ and $(k-r)!$ yields
$$
=\frac{n!}{\prod_{j=1}^r n_j! (n -\sum_{j=1}^r n_j)!}\prod_{j=1}^r (1 -\alpha)_{n_j-1} \frac{V_{n,k}}{k_{[r]}} S_{n-\sum_{j=1}^r n_j, k-r}^{-1, -\alpha},
$$
for $(x)_{[n]}= (x)(x-1)\cdots(x-n+1)$.
\end{prop}
\begin{coro} By a known result in \cite{gnepit06} for each model (\ref{EPPFgibbs}) the number of blocks $K_n$ has distribution
$$
\mathbb{P}(K_n=k)= V_{n,k}S_{n,k}^{-1, -\alpha},
$$
hence, conditioning (\ref{rmarg}) on $K_n=k$ yields
\begin{equation}
\label{marg_2}
\mathbb{P} (N_1=n_1, \dots, N_r=n_r| K_n=k)=  
\end{equation}
$$
\frac{n!}{\prod_{j=1}^r n_j! (n -\sum_{j=1}^{r} n_j)!}\frac{\prod_{j=1}^r (1 -\alpha)_{n_j-1} }{k_{[r]}} \frac{S_{n-\sum_{j=1}^r n_j, k-r}^{-1, -\alpha}}{S_{n,k}^{-1, -\alpha}}
$$
independently of the specific Gibbs model. For $r=k$ this is the general Gibbs analog of Eq. (41.8) in \cite{ewetav95}, and for $r=1$, $0 \leq k-1 \leq n-n_1$ and $n_1=1, \dots, n-k+1$
$$
\mathbb{P}(N_1=n_1| K_n=k)= {n \choose n_1} \frac{(1-\alpha)_{n_1-1}}{k} \frac{S_{n-n_1, k-1}^{-1, -\alpha, }}{S_{n,k}^{-1, -\alpha}}
$$
with expected value
$$
\mathbb{E}(N_1| K_n=k)= \frac{n}{k} \frac{S_{n-1, k-1}^{-1, -\alpha, -(1-\alpha)}}{S_{n,k}^{-1, -\alpha}}, 
$$
for $S_{n,k}^{-1, -\alpha, \gamma}$ generalized non-central Stirling numbers (see (\ref{noncentralsti}) in the Appendix).
Marginalizing (\ref{rmarg}) with respect to $K_n$ yields
$$
\mathbb{P} (N_1=n_1, \dots, N_r=n_r)= 
$$
$$=\frac{n!}{\prod_{j=1}^r n_j! (n -\sum_{j=1}^{r} n_j)!}\prod_{j=1}^r (1 -\alpha)_{n_j-1} \sum_{k-r=0}^{n -\sum_{j=1}^r n_j} \frac{V_{n,k} }{k_{[r]}} S_{n-\sum_{j=1}^r n_j, k-r}^{-1, -\alpha}. 
$$
\end{coro}
\subsection{Joint factorial moments of Gibbs sampling formulas}

Joint falling factorial moments for the Ewens' sampling formula (\ref{ewensEPPF}) of order $(r_1, \dots, r_n)$, for $r_l$ non negative integers and $n-\sum_l lr_l \geq 0$, are in \cite{ewetav95} (cfr. Eq. (41.9)) and correspond to 
$$
\mathbb{E}_{\theta} \left[\prod_{l=1}^n (C_{l,n})_{[r_l]}\right]= \frac{n!}{(n -\sum_{l=1}^n lr_l)!} \frac{(\theta)_{n -\sum_{l=1}^n lr_l}}{(\theta)_n} \prod_{l=1}^n \left(\frac{\theta }{l} \right)^{r_l}.
$$
Under the same conditions, the generalization to the $(\alpha, \theta)$ Poisson-Dirichlet partition model (\ref{2par}) has been obtained  is \cite{yamsib00} and is given by
$$
\mathbb{E}_{\alpha, \theta} \left[\prod_{l=1}^n (C_{l,n})_{[r_l]}\right]= \frac{n!}{(n -\sum_{l=1}^n lr_l)!} \frac{(\theta + \alpha)_{\sum_l r_l -1 \uparrow \alpha}}{(\theta +1)_{n-1}} \times
$$
$$\times \prod_{l=1}^n \left(\frac{(1- \alpha)_{l-1}}{l!}\right)^{r_l} (\theta + \alpha \sum_l r_l)_{n -\sum lr_l}.
$$
In the following Proposition we obtain the general result for the Gibbs sampling formula (\ref{sampling}) by resorting to a result in \cite{johkot05}, first established in \cite{dem18} then studied in \cite{jor67}. See also \cite{iye49, iye58}. 
\begin{prop}
\label{joint_sampl}Under a general $(\alpha, V)$ Gibbs partition model, joint falling factorial moments of the vector of counts $(C_{1,n}, \dots, C_{n,n})$ of order $(r_1, \dots, r_n)$ for $\sum_l lr_l  \leq n$ are given by
\begin{equation}
\label{jointprior}
\mathbb{E} \left[\prod_{l=1}^n (C_{l,n})_{[r_l]}\right]= \frac{n!}{\prod_{l=1}^n (l!)^{r_l}}
\frac{\prod_{l=1}^n \left[(1-\alpha)_{l-1}\right]^{r_l}}{(n -\sum_l l r_l)!} \sum_{k- \sum_l r_l=0}^{n- \sum_l l r_l}  V_{n,k} S_{n-\sum_l l r_l, k- \sum_l r_l}^{-1, -\alpha}
\end{equation} 
for $0 \leq k-\sum_{l=1}^n r_l \leq n-\sum_{l=1}^n lr_l$.
For $r_l=r \leq  \lceil{\frac nl \rceil}$ and $r_j=0$ for every $j \neq l$, the $r$-th falling factorial moment of $C_{l,n}$ results
\begin{equation}
\label{factmom_1}
\mathbb{E}\left[(C_{l,n})_{[r]}\right]= \frac{n! [(1-\alpha )_{l-1}]^r}{(l!)^r  (n-lr)!} \sum_{k-r=0}^{n-rl}
 V_{n,k} S_{n-rl, k-r}^{-1, -\alpha}.
\end{equation}
\end{prop}
\begin{proof} For $K_n=k$, let $C_{l,n}= \sum_{j=1}^k 1 (N_{j}=l)$. Then by a result for sum of non independent indicators  r.v.s in Johnson \& Kotz (2005, Sect. 10.2), or Charalambides (2005, Example 1.12), for $r \leq n$
\begin{equation}
\label{momentr}
\mathbb{E}\left[(C_l)_{[r]}\right]= \mathbb{E}(\sum_{j=1}^k {1\{N_j=l\}})_{[r]}= r! \sum_{(a_1, \dots, a_r)} \mathbb{P}(N_{a_1}=l, \dots, N_{a_r}=l),
\end{equation}
where the summation is extended over all $r$-combinations  $(a_1, \dots, a_r)$ of $\{1, \dots, k\}$.
Since in our case the number of blocks $K_n$ is random, and the vector $(N_1, \dots, N_r|K_n=k)$ is exchangeable then, for $l=1, \dots, n$,  
$$
\mathbb{E}\left[(C_l)_{[r]}\right]= \sum_{k-r=0}^{n-rl} \mathbb{E}\left[(C_l|K_n=k)_{[r]}\right] \mathbb{P}(K_n=k)=
$$
$$= \sum_{k-r=0}^{n-rl}r! {k \choose r} \mathbb{P}(N_1=l, \dots, N_r=l|K_n=k) \mathbb{P}(K_n=k)= 
$$
$$=\sum_{k-r=0}^{n-rl} r! {k \choose r} \mathbb{P}(N_1=l, \dots, N_r=l, K_n=k).
$$ 
hence
$$
\mathbb{E} \left[\prod_{l=1}^n (C_{l,n})_{[r_l]}\right]= \sum_{k- \sum_l r_l=0}^{n -\sum_{l} lr_l} (\prod_{l=1}^n r_l!) \frac{k!}{\prod_l r_l! (k- \sum_l r_l)! }\times
$$
\begin{equation}
\label{kalea}
\times \mathbb{P}(N_1=1, \dots, N_{r_1}=1, \dots, N_{\sum_l r_l -r_{n}+1}= n, \dots, N_{\sum_l r_l}= n, K_n=k).
\end{equation}
Inserting (\ref{rmarg}) in (\ref{kalea}) the result follows. \end{proof}

Notice that (\ref{jointprior}) generalizes the result in \cite{flp12a} Eq. (11), stated in terms of {\it generalized factorial coefficients}, (cfr. Eq. \eqref{chara} and \eqref{coeff}  in the Appendix), which corresponds to (\ref{factmom_1}). Next Proposition generalizes Proposition 2 (Dirichlet case) and  Proposition 4 (two parameter Poisson-Dirichlet case) in \cite{flp12a}. 
\begin{prop}
\label{dist_marg_uncon}
Under a general $(\alpha, V)$ Gibbs partition model, for each $n \geq 1$ the law of $C_{l,n}$, the number of blocks of size $l$, has distribution
\begin{equation}
\label{margL}
\mathbb{P}(C_{l,n}=x)=\frac{n! [(1-\alpha)_{l-1}]^x}{x!(l!)^{x}} \sum_{r=0}^{\lceil\frac nl\rceil-x} \frac{(-1)^{r}[(1-\alpha)_{l-1}]^r}{r! (l!)^r (n-rl-lx)!}\times
\end{equation}
$$
\times \sum_{k-r-x=0}^{n-rl-xl} V_{n,k} S_{n-rl-xl, k-r-x}^{-1, -\alpha}  
$$
for $x=0, \dots, \lceil{n/l}\rceil$, with expected value
\begin{equation}
\label{meanL}
\mathbb{E}(C_{l,n})= {n \choose l} (1 -\alpha)_{l-1} \sum_{k-1=0}^{n-l} V_{n,k}S_{n-l, k-1}^{-1, -\alpha}
\end{equation}
and the distribution of the number of singleton species $C_{1,n}$ follows from (\ref{margL})
\begin{equation}
\label{marg10}
\mathbb{P}(C_{1,n}=x)= \frac{n!}{x!} \sum_{r=0}^{n-x} \frac{(-1)^r}{r! (n-r-x)!} \sum_{k-r-x=0}^{n-r-x} V_{n,k} S_{n-r-x, k-r-x}^{-1, -\alpha}.
\end{equation}
with expected value
\begin{equation}
\label{mean_10}
\mathbb{E}(C_{1,n})= n \sum_{k-1=0}^{n-1} V_{n,k} S_{n-1, k-1}^{-1, -\alpha}.
\end{equation}
\end{prop}
\begin{proof}
(\ref{margL}) arises by the known relationship between discrete probability distributions and falling factorial moments (cfr.  (\ref{momprob}) in the Appendix). (\ref{meanL}) follows from (\ref{factmom_1}) for $r=1$. \eqref{marg10} and \eqref{mean_10} follow for $l=1$ and generalize (41.10) and (41.11) in \cite{ewetav95} to the entire Gibbs family.
\end{proof}

\section{Conditional multivariate Gibbs distributions}
The study of {\it conditional exchangeable random partitons}, i.e. random partitions starting with an initial allocation of the first $n$ natural integers in a certain number $k$ of blocks, has been initiated in \cite{lmp07} in view of proposing a Bayesian conditional nonparametric estimation of the richness of a population of species under {\it priors} on the unknown relative abundances belonging to the Gibbs class. (See also \cite{gne10}, Sect. 7). In \cite{cer09} it has been shown that the corresponding conditional partition probability function, describing the conditional allocation in {\it new} and {\it old} blocks of integers $n+1, n+2, \dots$, can be obtained by a multi-step variation of the classical {\it Chinese restaurant process} construction (CRP) for exchangeable partitions, (first devised by Dubins and Pitman, see \cite{pit06} Ch. 3). This variation helps to properly place the Bayesian nonparametric approach to species sampling problems under Gibbs priors into the Gnedin-Pitman's exchangeable random partitions theoretical framework. 
Here we recall the multi-step CRP for completeness.
\begin{prop} (Cerquetti, 2009) Given an infinite EPPF  model (\ref{EPPFgibbs}), assume that an unlimited numbers of groups of customers arrive sequentially in a restaurant with an unlimited numbers of circular tables, each capable of sitting an unlimited numbers of customers. Given the placement of the first group of $n$ customers in a ${\bf n}=(n_1, \dots, n_j)$ configuration in $j$ tables, a new {\it group} of $m \geq 1$ customers is \\\\
a) all seated at the $j$ old tables in configuration ${\bf m}=(m_1,\dots, m_j)$, for $m_i \geq 0$, $\sum_{i=1}^j m_i=m$, with probability 
\begin{equation}
\label{gibbsallold}
p_{\bf m}({\bf n})=\frac{V_{n+m,j}\prod_{i=1}^j (1-\alpha)_{n_i+m_i-1}}{V_{n,j}\prod_{i=1}^j (1-\alpha)_{n_i -1}}=\frac{V_{n+m,j}}{V_{n,j}} \prod_{i=1}^j (n_i-\alpha)_{m_i},
\end{equation}
b) all seated at $k$ {\it new} tables in configuration ${\bf s}=(s_1, \dots, s_{k})$, for $\sum_{i=1}^{k} s_i =m$, $1 \leq k \leq m$, $s_i \geq 1$, with probability  
\begin{equation}
\label{gibbsallnew}
p^{\bf s}({\bf n})=\frac{V_{n+m,j+k}\prod_{i=1}^j (1-\alpha)_{n_i -1}\prod_{i=1}^{k} (1-\alpha)_{s_i -1}}{V_{n,j} \prod_{i=1}^j (1-\alpha)_{n_i -1}}=\frac{V_{n+m, j+k}}{V_{n,j}}\prod_{i=1}^{k} (1-\alpha)_{s_i -1},\\
\end{equation}
c) a subset $s < m$ of the new customers is seated at $k$ {\it new} tables in configuration $(s_1,\dots,s_{k})$ and the remaining $m-s$ customers are seated at the {\it old} tables in configuration $(m_1,\dots, m_j)$ for $\sum_{i=1}^{j} m_i= m-s$, $1 \leq s \leq m$, $\sum_{i=1}^{k} s_i=s$, $m_i \geq 0$, $s_i \geq 1$ with probability
$$
p_{\bf m}^ {\bf s}({\bf n})=\frac{V_{n+m,j+k} \prod_{i=1}^j (1-\alpha)_{n_i+m_i-1}\prod_{i=1}^{k}(1-\alpha)_{s_i-1}}{V_{n,j}\prod_{i=1}^j (1-\alpha)_{n_i-1}}=
$$
which, by the multiplicative property of rising factorials (\ref{multiplicative}), simplifies to
\begin{equation}
\label{oldenew}
=\frac{V_{n+m,j+k}}{V_{n,j}}\prod_{i=1}^j (n_i-\alpha)_{m_i}\prod_{i=1}^{k}(1-\alpha)_{s_i-1}.
\end{equation}	
\end{prop}
Now, as in \cite{lpw08}, given the allocation of the first $n$ integers in $j$ blocks with multiplicities $(n_1, \dots, n_j)$, let $K_m$ be the number of {\it new} blocks generated by the additional $m$ integers, $(S_1, \dots, S_{K_m})$ the vector of the sizes of the {\it new} blocks in exchangeable random order and $S_{m}=\sum_{i=1}^{K_m} S_{i}$ the total number of {\it new} integers in {\it new} blocks. To obtain the joint conditional distribution of the vector $(K_m, S_m, S_1, \dots, S_{K_m})$ of the number and multiplicities of new blocks, and total observations in new blocks, it is enough to marginalize (\ref{oldenew}) with respect to all $(m_1, \dots, m_j)$ allocations of  $m-S_m$ observations in {\it old} blocks, and to multiply for the combinatorial coefficient accounting  for the number of partitions of $[m]$ providing the same sizes and the same number  $k$ of new blocks and the same number $s$ of integers in new blocks. We can hence state the following. 

\begin{prop} Under a general $(\alpha, V)$ Gibbs partition model the joint conditional distribution of $(S_m, K_m, S_1, \dots, S_{K_m})$, for $S_1, \dots, S_{K_m}$ in {\it exchangeable random order}, given the initial allocation of $n$ integers in $j$ blocks, corresponds to 
$$
\mathbb{P} (K_m=k, S_m=s, S_{1}= s_1, \dots, S_{K_m}=s_k| n_1, \dots, n_{j})=
$$
$$= \frac{s!}{s_1! \cdots s_k! k!} \frac{V_{n+m, j+k}}{V_{n,j}} {m \choose s} (n -j\alpha)_{m-s} \prod_{i=1}^{k} (1 -\alpha)_{s_i -1}=$$
or alternatively
\begin{equation}
\label{newblock}
=\frac{m!}{s_1! \cdots s_k! k! m-s!}  \frac{V_{n+m, j+k}}{V_{n,j}} (n -j\alpha)_{m-s} \prod_{i=1}^{k} (1 -\alpha)_{s_i -1}. 
\end{equation}
Moreover, conditioning on $S_m$, by Eq. (11) in \cite{lpw08}, yields
\begin{equation}
\label{uffa_2}
\mathbb{P} (K_m= k,  S_1=s_1, \dots, s_{K_m}=s_k| K_n=j, S_m=s)= 
\end{equation}
$$
=\frac{s!}{s_1! \cdots s_k! k!} \frac{V_{n+m, j+k}}{\sum_{i=0}^{s} V_{n+m, j+i} S_{s, i}^{-1, -\alpha}} \prod_{i=1}^k (1 -\alpha)_{s_i-1},
$$
while conditioning on $K_m$, by Eq. (4) in \cite{lmp07}, eliminates the dependency on the specific $(V_{n,k})$ Gibbs model as in \eqref{marg_2} 
\begin{equation}
\label{uffa_3}
\mathbb{P} (S_{1}= s_1, \dots, S_{K_m}=s_k| K_{m}=k, S_{m}=s, K_n=j)=
\end{equation}
$$= \frac{s!}{s_1! \cdots s_k! k!} \frac {\prod_{i=1}^k (1-\alpha)_{s_i -1}}{S_{s,k}^{-1, -\alpha}}.
$$
\end{prop}

\begin{rema} 
Further results for the conditional moments of any order of $K_m$  and for the conditional asymptotic distribution of a proper normalization of $K_m$  under $(\alpha, \theta)$ Poisson-Dirichlet partition models are in \cite{flp09}. A simplified approach to the posterior analysis of the two-parameter model exploiting the {\it deletion of classes property} and the Beta-Binomial distribution of $S_m|K_n=j$ is in \cite{cer11a}. A general result for {\it conditional $\alpha$ diversity} for Poisson-Kingman partition models driven by the stable subordinator (\cite{pit03}) has been obtained in \cite{cer11b}.

\end{rema}

\begin{rema} Notice that equations (\ref{newblock}), (\ref{uffa_2}), and (\ref{uffa_3}) fix corresponding formulas (9), (19) and (34) in \cite{lpw08} which are missing the combinatorial coefficients. The problem in Lijoi et al. (2008) seems to follow from some confusion between conditional Gibbs EPPFs, as arising from the multistep sequential construction of Proposition 4, and joint conditional distributions of the corresponding random vectors. We stress here that an EPPF provides the probability of a particular partition characterized by a certain allocation in a certain number of blocks with certain multiplicities. This differs from the probability of the random vector of the multiplicities to assume that specific value, which is obtained by summing over all different partitions providing the same multiplicities in the same number of blocks. The results in the following sections show that once the corrected formulas for the joint conditional distribution are properly identified, the derivation of estimators for quantities of interest  in Bayesian nonparametric  species sampling modeling simply follows by working with joint conditional marginals of  (\ref{condmulti}).
\end{rema}

The first step is to define the conditional analog of (\ref{prior}).
\begin{defn} Under a general $(\alpha, V)$ Gibbs partition model, the multivariate distribution of the vector $(S_1, \dots, S_{K_m}, K_m)$, which arises by marginalizing (\ref{newblock}) with respect to $S_m$,   
\begin{equation}
\label{condmulti}
\mathbb{P}_{\alpha, V} (S_1=s_1, \dots, S_{K_m}=s_k, K_m=k| {\bf n})=
\end{equation}
$$= \frac{m!}{s_1! \cdots s_{k}! k! } \frac{V_{n+m, j+k}}{V_{n,j}}  \sum_{s=k}^{m} \frac{(n-j\alpha)_{m-s}}{(m-s)!} \prod_{i=1}^k (1 -\alpha)_{s_i -1}
$$
for $(s_1, \dots, s_{k}): \sum_j{s_j} \in [k, m]$, $k \in [1, m]$, is termed {\it conditional multivariate Gibbs distribution of parameters $(m, \alpha, j, n)$}, for $m \geq 1$, $\alpha \in (-\infty, 1)$ and $j \leq n$.
\end{defn}

In the next Proposition, mimicking the technique adopted in the previous section for the unconditional case, we derive marginals of (\ref{condmulti}) as the tools to obtain joint conditional falling factorial moments of the {\it conditional Gibbs sampling formula}. For $W_{l,m}= \sum_{i=1}^{K_m} 1 \{S_i=l\}$ this is given by the usual change of variable in (\ref{condmulti}), hence
\begin{equation}
\label{condsampl}
\mathbb{P}(W_{1,m}=w_1, \dots, W_{m,m}=w_m|n_1, \dots, n_j)= 
\end{equation}
$$=\frac{m!V_{n+m, j+k}}{V_{n,j}} \sum_{s=k}^{m} \frac{(n-j\alpha)_{m-s}}{(m-s)!} \prod_{i=1}^s
\frac{[(1-\alpha)_{i-1}]^{w_i}}{(i!)^{w_i} w_i!}.
$$
In what follows we will resort to the convolution relation which defines {\it non-central} generalized Stirling numbers in terms of {\it central} generalized Stirling numbers
\begin{equation}
\label{convo_1}
S_{n,k}^{-1, -\alpha, \gamma}= \sum_{s=k}^{n} {n \choose s} S_{s,k}^{-1, -\alpha} (-\gamma)_{n-s},
\end{equation}
see the Appendix (cfr. (\ref{convo})) for further details. 
\begin{prop} 
\label{marg_new_cond}
Under a general $(\alpha, V)$ Gibbs partition model the r-dimensional marginal of (\ref{condmulti}), for $(s_1, \dots, s_r): \sum_i s_i \leq s \leq m$ and $0 \leq k-r \leq m -\sum_{i=1}^r s_i$, is given by
\begin{equation}
\label{margnew}
\mathbb{P} (S_1=s_1, \dots, S_r=s_r, K_m= k| {\bf  n} )=
\end{equation}
$$
= \frac{m! [\prod_{i=1}^r (1-\alpha)_{s_i-1}]}{\prod_{i=1}^r s_i!  (m-\sum_{i=1}^r s_i)!} \frac{(k-r)!}{ k!} \frac{V_{n+m, j+k}}{V_{n,j}} S_{m -\sum_{i=1}^r s_i, k- r}^{-1, -\alpha, -(n-j\alpha)}.
$$
\end{prop}

\begin{proof}
Multiplying and dividing (\ref{condmulti}) by $(s-\sum_{i=1}^r s_i)!$ and $(m-\sum_{i=1}^r s_i)!$ and marginalizing yields
$$
\mathbb{P} (S_1=s_1, \dots, S_r=s_r, K_m= k| {\bf  n} )=\frac{m! [\prod_{i=1}^r (1-\alpha)_{s_i-1}]}{\prod_{i=1}^r s_i!  (m-\sum_{i=1}^r s_i)!} \frac{ 1}{ k! } \frac{V_{n+m, j+k}}{V_{n,j}} \times
$$
$$\times \sum_{s-\sum_{i=1}^r s_i= k-r}^{m-\sum_{i=1}^r s_i} \frac{(m-\sum_{i=1}^r s_i)! (n-j\alpha)_{m-s}}{(s-\sum_{i=1}^r s_i)! (m-s)!}  \sum_{(b_1, \dots, b_{k-r})} \frac{ (s-\sum_{i=1}^r s_i)!}{\prod_i b_i!} \prod_{i=1}^{k-r} (1 -\alpha)_{b_i -1}=
$$
further multiplying and dividing by $(k -r)!$ we obtain
$$
=\frac{m! [\prod_{i=1}^r (1-\alpha)_{s_i-1}]}{\prod_{i=1}^r s_i!  (m-\sum_{i=1}^r s_i)!} \frac{(k-r)!}{ k!} \frac{V_{n+m, j+k}}{V_{n,j}} \times
$$
$$
\times \sum_{s-\sum_{i=1}^r s_i= k-r}^{m-\sum_{i=1}^r s_i} {m-\sum_{i=1}^r s_i  \choose s- \sum_{i=1}^r s_i} (n -j\alpha)_{m-s} S_{s-\sum_{i=1}^r s_i, k-r}^{-1, -\alpha},
$$
and the result follows by (\ref{convo_1}).
\end{proof}
The following Proposition generalizes Theorem 2. in \cite{flp12a}. We adopt the notation $W_{l,m}^{(n)}$ to indicate components of (\ref{condsampl})
\begin{prop}
\label{joint_cond2}
Under a general $(\alpha, V)$ Gibbs partition model, joint falling factorial moments of order $(r_1, \dots, r_m)$ of the conditional sampling formula (\ref{condsampl}) arise by an application of (\ref{margnew}) in (\ref{kalea}). For $m- \sum_l lr_l \geq 0$
\begin{equation}
\label{jointnew}
\mathbb{E} \left[ \prod_{l=1}^m (W_{l,m}^{(n)})_{(r_l)}\right]=
\end{equation}
$$
= \frac{m! \prod_{l=1}^m \left[(1-\alpha)_{l-1}\right]^{r_l}}{(m- \sum_{l} l r_l)! \prod_l (l!)^{r_l}} \frac{1}{V_{n,j}} \sum_{k - \sum_l r_l =0}^{m- \sum_l l r_l} V_{n+m, j+k} S_{m -\sum_l l r_l, k- \sum_l r_l}^{-1, -\alpha, -(n-j\alpha)}.
$$
For $r_l=r \leq  \lceil{\frac ml \rceil}$ and $r_j=0$ for $j\neq l$ then 
\begin{equation}
\label{mmnew}
\mathbb{E}[(W_{l,m}^{(n)})_{[r]}]=\frac{m!}{(m-rl)!}  \frac{[(1-\alpha)_{l-1}]^r}{(l!)^r} \frac{1}{V_{n,j}} \sum_{k-r=0}^{m-rl} V_{n+m, j+k} S_{m-rl, k-r}^{-1, -\alpha, -(n- j\alpha)},
\end{equation}
which agrees with the result in Theorem 2. in \cite{flp12a} expressed in terms of non central generalized factorial numbers (cfr. (\ref{coeffsti}) in the Appendix).
\end{prop}
\begin{proof}
By the analogy between (\ref{margnew}) and (\ref{rmarg}) the proof moves along the same lines as the proof of Proposition \ref{joint_sampl}, exploiting the marginals obtained in Proposition \ref{marg_new_cond}.
\end{proof}
\begin{rema}
Notice the great computational advantage provided by the technique based on marginals of multivariate Gibbs distributions devised in the previous section with respect to the complexity of the approach adopted in \cite{flp12a}. (\ref{jointnew}) immediately follows as the conditional analog of the result obtained in Proposition \ref{joint_sampl} without any need to provide a new proof.  
\end{rema}

In \cite{flp12a}, (cfr. Propositions 6 and 9), explicit marginals of (\ref{condsampl}) have been derived for the Dirichlet $(\theta)$, the $(\alpha, \theta)$ Poisson-Dirichlet and the Gnedin-Fisher $(\gamma)$ (\cite{gne10}) partition models, . In the next Proposition we obtain the general result for the entire Gibbs family thus providing the conditional analog of Proposition \ref{dist_marg_uncon}.  
\begin{prop} Under a general $(\alpha, V)$ Gibbs partition model the marginal distribution of (\ref{condsampl}) for $x=0, \dots, \lceil  m/l \rceil$
corresponds to 
\begin{equation}
\label{margcondw}
\mathbb{P}(W_{l,m}^{(n)}= x)= \frac{[(1-\alpha)_{l-1}]^x}{x! (l!)^x}\frac{m!}{V_{n,j}}\times
\end{equation}
$$\times \sum_{r= 0}^{\lceil{\frac ml \rceil}-x}\frac{(-1)^r [(1-\alpha)_{l-1}]^r}{r! (l!)^r (m -rl-xl)!} \sum_{k-r-x=0}^{m-rl-xl} V_{n+m, j+k} S_{m-rl-xl, k-r-x}^{-1, -\alpha, -(n-j\alpha)}.
$$ 
Its expected value, which provides the Bayesian nonparametric estimator under quadratic loss function, for the number of {\it new} species represented $l$ times, arises from (\ref{mmnew}) for $r=1$ 
\begin{equation}
\label{mean_cond}
\mathbb{E}(W_{l,m}^{(n)})={m \choose l} \frac{(1-\alpha)_{l-1}}{V_{n,j}} \sum_{k-1=0}^{m-l} V_{n+m, j+k} S_{m-l, k-1}^{-1, -\alpha, -(n-j\alpha)}.
\end{equation}
The conditional distribution of the number of  new singleton species $W_{1,m}^{(n)}$ will be 
$$
\mathbb{P}(W_{1,m}^{(n)}=x)=\frac{m!}{x!}\frac{1}{V_{n,j}} \sum_{r=0}^{m-x} \frac{(-1)^r}{r! (m-r-x)!} \sum_{k-r-x=0}^{m-r-x} V_{n+m, j+k} S_{m-r-x, k-r-x}^{-1, -\alpha, -(n-j\alpha)}
$$
and a Bayesian estimator of the number of new singleton species follows from (\ref{mean_cond}) as
$$
\mathbb{E}(W_{1,m}^{(n)})= \frac{m}{V_{n,j}} \sum_{k-1=0}^{m-1} V_{n+m, j+k} S_{m-1,k-1}^{-1, -\alpha, -(n-j\alpha)}.
$$
\end{prop}
\begin {proof}
(\ref{margcondw}) arises by an application of (\ref{momprob}), (\ref{mean_cond}) follows from (\ref {mmnew}) for $r=1$ and corresponds to Eq. (17) in \cite{flp12a} expressed in terms of generalized non central factorial numbers (cfr. (\ref{coeffsti}) in the Appendix).
\end{proof}
\section{Multivariate P\'olya-Gibbs distributions} In this Section we focus on the conditional random allocation of the additional $m$ integers in the $j$ {\it old} blocks. First we derive the conditional joint distribution of the random vector $(M_{1,m}, \dots, M_{j,m}, S_m)$ of the sizes of the $m-S_m$ observations falling in the $j$ {\it old} blocks and of the total number of {\it new} observations  $S_m$ falling in {\it new} blocks. Then, similarly to the previous sections, we move attention to the corresponding vector of counts and its joint falling factorial moments. From (\ref{oldenew}), marginalizing with respect to the partitions in new blocks, and multiplying for the combinatorial coefficient accounting for the number of allocations providing the same sizes of {\it old} blocks and the same number of total observations in {\it new} blocks, we obtain
\begin{equation}
\label{jointold}
\mathbb{P}(M_{1,m}=m_1, \dots, M_{j,m}=m_j, S_m=s|{n_1, \dots, n_j})= 
\end{equation}
$$=\frac{m!}{\prod_{i=1}^j m_i! s!} \prod_{i=1}^j (n_i -\alpha)_{m_i}\sum_{k=0}^s \frac{V_{n+m, j+k}}{V_{n,j}} S_{s, k}^{-1, -\alpha}, 
$$
for  $m_{i} \geq 0$ for $i=1, \dots, j$ and $\sum_{i=1}^j m_{i}=m -S_m$.
\begin{rema}
Since the number of old blocks is fixed, (\ref{jointold}) may be interpreted as a generalization of {\it multivariate P\'olya distributions}. If $Q_{V_{n,k}}$ is the conditional law,  given $(n_1, \dots, n_j)$, of the vector $(\tilde{P}_{1,n}, \dots, \tilde{P}_{j,n}, R_{j,n})$, for $\tilde{P}_{j,n}= \tilde{P}_j|n_1, \dots, n_j$ the conditional random relative abundance of the $j$-th species to appear, and  $R_{j,n}=1 -\sum_{i=1}^j \tilde{P}_{i,n}$, then (\ref{jointold}) turns out to be a $Q_{V}$-multinomial mixture that we term {\it multivariate P\'olya-Gibbs distribution} of parameters $(n_1-\alpha, \dots, n_j- \alpha, V)$. Moreover $Q_{V}$ will be the limit law, for $m \rightarrow \infty$, of the random vector
$$
\frac{M_{1,m}^{(n)}}{m}, \dots, \frac{M_{j,m}^{(n)}}{m}, \frac{S_m}{m},
$$
where $M_{i, m}^{(n)}$ stands for a component of \eqref{jointold}.
Notice that for the two-parameter Poisson-Dirichlet $(\alpha, \theta)$ model, by a result in \cite{pit96}, (cfr. Sect. 3.7, Corollary 20),
$$(\tilde{P}_{1,n}, \dots, \tilde{P}_{j,n}, R_{j,n})\sim Dir[n_1 -\alpha, \dots, n_j-\alpha, \theta +j\alpha],$$
and substituting $V_{n,k}= (\theta +\alpha)_{k-1 \uparrow \alpha} / (\theta +1)_{n-1}$ in (\ref{jointold}) yields
\begin{equation}
\label{jointoldPD}
\mathbb{P}_{\alpha, \theta}(M_{1,m}=m_1, \dots, M_{j,m}=m_j, S_m=s|{\bf n})= $$
$$
=\frac{m!}{\prod_{i=1}^j m_i! s!} \frac{\prod_{i=1}^j (n_i -\alpha)_{m_i} (\theta +j\alpha)_{s}}{(n +\theta)_{m}}
\end{equation}
which is a proper {\it multivariate P\'olya distribution} of parameters $(m,  n_1-\alpha, \dots, n_j-\alpha, \theta+j\alpha)$.
\end{rema}
Next Proposition provides the general marginal that we need to obtain joint falling factorial moments of the vector of counts corresponding to \eqref{jointold}.
\begin{prop} Under a general $(\alpha, V)$ Gibbs model, the conditional joint marginal distribution of the vector of the sizes $(M_{1,m} \dots, M_{r,m})$ of the  additional {\it new} observations falling in the first $r$ {\it old} blocks corresponds to
\begin{equation}
\label{margvecchi}
\mathbb{P}(M_{1,m}=m_1, \dots, M_{r,m}=m_r|n_1, \dots, n_j)=
\end{equation}
$$=\frac{m! \prod_{i=1}^r (n_i -\alpha)_{m_i} }{\prod_{i=1}^r m_i! (m- \sum_{i=1}^r m_i)!} \sum_{k=0}^{m-\sum_{i=1}^r m_i} \frac{V_{n+m, j+k}}{V_{n,j}} S_{m- \sum_{i=1}^r m_i, k}^{-1, -\alpha, -(n-(j-r)\alpha -\sum_{i=1}^r n_i)}.
$$
\end{prop}
\begin{proof}By (\ref{jointold}), the jont marginal of the first $r$ blocks and  $S_m$ is easily obtained as
$$
\mathbb{P}(M_{1,m}=m_1, \dots, M_{r,m}=m_r, S_m=s|{\bf n})=\frac{m!}{\prod_{i=1}^r m_i!  (m -s-\sum_{i=1}^r m_i)! s!}\times $$
$$\times \prod_{i=1}^r (n_i -\alpha)_{m_i} (n- j\alpha - \sum_{i=1}^r n_i +r\alpha)_{m-s-\sum_{i=1}^r m_i} \sum_{k=0}^s \frac{V_{n+m, j+k}}{V_{n,j}} S_{s, k}^{-1, -\alpha},
$$
marginalizing with respect to $S_m$, and multiplying and dividing by $(m -\sum_{i=1}^r m_i)!$ yields 
$$
\mathbb{P}(M_{1,m}=m_1, \dots, M_{r,m}=m_r|n_1, \dots, n_j)
=\sum_{s=0}^{m -\sum_{i=1}^r m_i} {m - \sum_{i=1}^r m_i \choose s} \times$$
$$\times \frac{m! \prod_{i=1}^r (n_i -\alpha)_{m_i} }{\prod_{i=1}^r m_i! (m- \sum_{i=1}^r m_i)!} (n- j\alpha - \sum_{i=1}^r n_i +r\alpha)_{m-s-\sum_{i=1}^r m_i} \sum_{k=0}^s \frac{V_{n+m, j+k}}{V_{n,j}} S_{s, k}^{-1, -\alpha}=
$$
which reduces to
$$
=\frac{m! \prod_{i=1}^r (n_i -\alpha)_{m_i} }{\prod_{i=1}^r m_i! (m- \sum_{i=1}^r m_i)!} \sum_{k=0}^{m-\sum_{i=1}^r m_i} \frac{V_{n+m, j+k}}{V_{n,j}}\sum_{s=k}^{m-\sum_{i=1}^{r} m_i} {m - \sum_{i=1}^r m_i \choose s}  \times
$$
$$
\times(n- j\alpha - \sum_{i=1}^r n_i +r\alpha)_{m-s-\sum_{i=1}^r m_i}S_{s, k}^{-1, -\alpha},
$$
and the result follows by an application of (\ref{convo_1}).
\end{proof}
Now let $O_{l,m}^{(n)}=\sum_{i: n_i \leq l} 1\{n_i+M_{i,m}=l| n_1, \dots, n_j\}$, for $l=1, \dots, n+m$, be the number of {\it old} blocks of size $l$ after the allocation of the additional $m$-sample,  then,  to obtain the joint falling factorial moments of any order for the sampling formula of (\ref{jointold}) we exploit the multivariate version of the result (\ref{momentr}) recalled in the proof of Proposition 2, namely
\begin{equation}
\label{johgen}
\mathbb{E}\left[(O_{l,m}^{(n)})_{[r]}\right]= r! \sum_{(\xi_1, \dots, \xi_r)} \mathbb{P} (M_{\xi_1}= l-n_1, \dots, M_{\xi_r}=l-n_r).
\end{equation}
For  $m_i=l-n_i$, (\ref{margvecchi}) specializes as
\begin{equation}
\label{oldmarg}
\mathbb{P}(M_{1,m}=l-n_1, \dots, M_{r,m}=l-n_r|n_1, \dots, n_j)=
\end{equation}
$$=\frac{m! \prod_{i=1}^r (n_i -\alpha)_{l -n_i} }{\prod_{i=1}^r (l-n_i)! (m- rl +\sum_{i=1}^r n_i)!}\frac{1}{V_{n,j}} \times
$$
$$\times \sum_{k=0}^{m-lr + \sum_{i=1}^r n_i} V_{n+m, j+k} S_{m- lr +\sum_{i=1}^r n_i, k}^{-1, -\alpha, -(n-(j-r)\alpha -\sum_{i=1}^r n_i)},
$$
and the one-dimensional marginal of (\ref{oldmarg}) corresponds to 
\begin{equation}
\label{onemargold}
\mathbb{P}(M_{i,m}=l-n_i|{\bf n})= {m \choose {l-n_i}}\frac{ (n_i-\alpha)_{l-n_i}}{V_{n,j}} \sum_{k=0}^{m-l+n_i} V_{n+m, j+k} S_{m-l +n_i,k}^{-1, -\alpha, -(n-j\alpha+\alpha -n_i)}.
\end{equation}

The following result easily follows from (\ref{johgen}) as the analog of Propositions \ref{joint_sampl} \and \ref{joint_cond2}.


\begin{prop} Under a general $(\alpha, V)$ Gibbs model, the joint falling factorial moments of the vector of the number of old blocks of different size $(O_{1,m}^{(n)}, \dots, O_{n+m,m}^{(n)})$, after the allocation of the additional $m$-sample, given the initial allocation  $n_1, \dots, n_k$ is given by
$$
\mathbb{E}\left[ (\prod_{l=1}^{n+m}(O_{l,m}^{(n)})_{[r_l]})\right]= 
$$
$$
=\prod_{l=1}^{n+m} r_l! \sum_{({\bf \Xi}_{r_1}, \dots, {\bf \Xi}_{r_{n+m}})} \frac{m! \prod_{l=1}^{n+m} \prod_{i=1}^{r_l} (n_{\xi_i}- \alpha)_{l-n_{\xi_i}}}{\prod_{l=1}^{n+m} \prod_{i=1}^{r_l} (l- n_{\xi_i})! (m -\sum_l lr_l + \sum_l \sum_{i=1}^{r_l} n_{\xi_i})!} \times
$$
$$
\times\sum_{k=0}^{m -\sum_l l r_l + \sum_l \sum_{i=1}^{r_l} n_{\xi_i}} \frac{V_{n+m, j+k}}{V_{n,j}} S_{m -\sum_l l r_l+ \sum_l \sum_{i=1}^{r_l} n_{\xi_i}, k}^{-1, -\alpha, -(n -(j -\sum_l r_l)\alpha - \sum_l \sum_i n_{\xi_i} )},
$$
for ${\bf \Xi}_{r_1}=(\xi_1, \dots, \xi_{r_1}), \dots, {\bf \Xi}_{r_{n+m}}= (\xi_{\sum_l r_l - r_{n+m}}, \dots, \xi_{\sum_{l=1}^{n+m} r_l}$), $\xi_i: n_{\xi_i} \leq l$, and each ${\bf \Xi}_{r_l}$ ranges over all the combinations of $r_l$ elements of $j$. For $r_l=r$ and $r_j=0$ for $j\neq l$, then
\begin{equation}
\label{momoldcond}
\mathbb{E}\left[(O_{l,m}^{(n)})_{[r]}\right]= r! \sum_{(\xi_1, \dots, \xi_r)}\frac{m! \prod_{i=1}^r (n_{\xi_i} -\alpha)_{l -n_{\xi_i}} }{\prod_{i=1}^r (l-n_{\xi_i})! (m- rl+ \sum_{i=1}^r n_{\xi_i})!} \times
\end{equation}
$$
\times\sum_{k=0}^{m-lr +\sum_{i=1}^r n_{\xi_i}} \frac{V_{n+m, j+k}}{V_{n,j}} S_{m- lr +\sum_{i=1}^r n_{\xi_i}, k}^{-1, -\alpha, -(n-(j-r)\alpha -\sum_{i=1}^r n_{\xi_i})}
$$
for $\xi_i:n_{\xi_i} \leq l$, which agrees with the result in Theorem 1. in \cite{flp12a}.
\end{prop}
Next Proposition generalizes the results in Proposition 5 (two-parameter Poisson-Dirichlet case) and Proposition 9 (one parameter Gnedin-Fisher case \cite{gne10}) in \cite{flp12a} to the entire $(\alpha, V)$ Gibbs family.

\begin{prop} Under a general $(\alpha, V)$ Gibbs model, from (\ref{momoldcond}) and (\ref{momprob}), the conditional marginal law of  $O_{l,m}^{(n)}$ is given by 
$$
\mathbb{P}(O_{l,m}^{(n)}=y)= \sum_{r=0}^{\lceil{\frac{m -rl +\sum_{i=1}^{r+y} n_{\xi_i}}{l} \rceil} -y} \frac{(-1)^r (r+y)!}{y! r!} \frac{1}{V_{n,j}} \times
$$
$$\times \sum_{(\xi_1, \dots, \xi_{r+y})} \frac{m!}{\prod_{i=1}^{r+y} (l-n_{\xi_i})! (m -rl-yl+ \sum_{i=1}^{r+y} n_{\xi_i})! }{\prod_{i=1}^{r+y} (n_{\xi_i}- \alpha)_{l-n_{\xi_i}}}\times
$$
$$
\times \sum_{k=0}^{m- rl-ly+\sum_{i=1}^{r+y} n_{\xi_i}}  V_{n+m, j+k} S_{m-lr-ly+\sum_{i=1}^{r+y} n_{\xi_i}, k}^{-1, -\alpha, -(n-(j-r-y)\alpha- \sum_{i=1}^{r+y} n_{\xi_i})}.
$$
Its expected value, which plays the role of the Bayesian nonparametric estimator, under quadratic loss function, of the number of old species represented $l$ times, follows from (\ref{onemargold}) as
\begin{equation}
\label{old_marg}
\mathbb{E}(O_{l,m}^{(n)})= 
\mathbb{E}\left({\sum_{i:n_i \leq l} 1(n_i+M_{i,m}=l|n_1, \dots, n_j)}\right)=
\end{equation}
$$
= \sum_{i: n_i \leq l} \mathbb{E}(1(M_{i,m}=l-n_i|n_1, \dots, n_j))= \sum_{i:n_i \leq l} \mathbb{P}(M_{i,m}=l-n_i|n_1, \dots, n_j)= $$
$$=\sum_{i:n_i \leq l} {m \choose {l-n_i}} \frac{(n_i-\alpha)_{l-n_i}}{V_{n,j}} \sum_{k=0}^{m-l+n_i} V_{n+m,j+k} S_{m-l+n_i, k}^{-1, -\alpha, -(n -j\alpha+\alpha-n_i)},
$$
or from \eqref{momoldcond} for $r=1$ and agrees with Eq. (15) in \cite{flp12a}.
\end{prop}
\begin{rema} Relying on the technique presented in this paper, falling factorial $r$-th moments of $Z_{l,m}^{(n)}$, the total number of {\it old} and {\it new} blocks of size $l$ after the allocation of the additional $m$-sample, as derived in Th. 3 in \cite{flp12a} by means of a very complex procedure, may be obtained in a straightforward way by the full conditional joint distribution 
$$
\small{\mathbb{P}(S_1=s_1, \dots, S_k=s_k, S_{m}=s, K_m=k, { M}_{1, m}={m}_{1}, \dots, M_{j,m}=m_{j}| {\bf n})=}
$$
$$
=\frac{m!}{\prod_{i=1}^k s_i! k! \prod_{i=1}^j m_i! s!} \frac{V_{n+m, j+k}}{V_{n,j}} \prod_{i=1}^k (1 -\alpha)_{s_i-1}\prod_{i=1}^j (n_i-\alpha)_{m_i}.
$$
Multiplying for the way to choose $t$ blocks among the {\it old} and $r-t$ among the {\it new} for every $t$, from (\ref{mmnew}) and (\ref{momoldcond}) we get
$$
\mathbb{E}\left[(Z_{l,m}^{(n)})_{[r]}\right]=
$$
$$=\sum_{t=0}^r {r \choose t} t! \sum_{(\xi_{i_1}, \dots, \xi_{i_t})}\frac{m!  [(1 -\alpha)_{l-1}]^{r-t} \prod_{i=1}^t (n_{\xi_i}-\alpha)_{l-n_{\xi_i}}  }{\prod_{i=1}^t (l-n_{\xi_i})! (l!)^{r-t} (m - tl +\sum_{i=1}^t n_{\xi_i} - (r-t)l)!} \times
$$
$$\times \sum_{k-r+t=0}^{m - rl +\sum n_{\xi_i}} \frac{V_{n+m, k+j}}{V_{n,j}} S_{m - rl +\sum n_{\xi_i}, k - r+t}^{-1, -\alpha, -(n -j\alpha -\sum n_{\xi_i} + t\alpha)}
$$
which agrees with Theorem 3. in \cite{flp12a}.
\end{rema}

In the next section we provide one more example of the importance of working with marginals of conditional multivariate Gibbs distributions in the implementation of the Bayesian nonparametric approach to species sampling problems under Gibbs priors. 

\section {Bayesian nonparametric estimation of the probability to observe a species of a certain size}
In species sampling problems, particularly in ecology or genomics, given a basic sample $(n_1, \dots, n_j)$, interest may be in estimating the {\it probability} to observe at step $n+m+1$ a species already represented $l$ times both belonging to an {\it old} species or to a {\it new} species eventually arising in the $m$-additional sample which is still not observed. This is the topic of a recent paper by Favaro {\it et al.} (2012b) and can be seen as a generalization of the problem of estimating the {\it discovery probability}, i.e. the probability to discover a {\it new} species, not represented in the previous $n+m$ observations.  A Bayesian nonparametric estimator of the discovery probability under general $(\alpha, V)$ Gibbs partition models has been first derived in \cite{lmp07}.

In this Section we show that working with marginals of conditional Gibbs multivariate distributions  greatly simplifies the derivation of the results obtained in \cite{flp12b}, thus providing another example of the importance of the technique proposed in this paper. \\

First recall that by sequential construction of exchangeable partitions, the probability to observe an {\it old} species observed $l$ times in the basic $n$-sample at observation $n+1$, easily follows by one-step prediction rules for general Gibbs EPPFs (see e.g. \cite{pit06}). For $c_{l,n}= \sum_{i=1}^j 1\{n_i=l\}$, for $l=1, \dots, n$ then
$$
p_{l,n}(n_1, \dots, n_j)=c_{l,n}\frac{p(n_1, \dots, l+1, \dots, n_j)}{p(n_1, \dots, l , \dots, n_j)}=c_{l,n} \frac{V_{n+1, j}}{V_{n,j}} (l -\alpha).
$$ 

Given a basic sample $(n_1, \dots, n_j)$, but assuming as in \cite{flp12b} an intermediate $m$-sample still to be observed, the probability to observe a species represented $l$ times among {\it new} species at observation $n+m+1$ will be a random variable, namely
\begin{equation}
\label{newelle}
P^{n+m+1}_{new,l}(\alpha, V)=\frac{V_{n+m+1, j+K_m}}{V_{n+m, j +K_m}} (l- \alpha)W_{l, m}^{(n)},
\end{equation}
for $K_m$ the random number of {\it new} species induced by the additional sample and $W_{l,m}^{(n)}$ the random number of species represented $l$ times in the additional sample given the basic sample. 

In the following Proposition we show how the Bayesian nonparametric estimator, under quadratic loss function, of (\ref{newelle}), (cfr. Theorem 2. in  \cite{flp12b}), may be obtained in few elegant steps.

\begin{prop}Under a general $(\alpha, V)$ Gibbs partition model, for $W_{l,m}^{(n)}= \sum_{i=1}^{K_m}1 \{S_i=l|K_n=j\}$ the Bayesian nonparametric estimator of $P^{m+n+1}_{new,l}(\alpha, V)$ is given by 
\begin{equation}
\label{newdisco}
\mathbb{E}_{(S_1,\dots, S_{K_m}, K_m| K_n=j)} \left( \frac{V_{n+m+1, j+K_m}}{V_{n+m, j +K_m}} (l- \alpha)W_{l, m}^{(n)}\right)=  
$$
$$
=(l-\alpha) \sum_{k-1=0}^{m-l} \frac{V_{n+m+1, j+k}}{V_{n,j}} {m \choose l} (1 -\alpha)_{l-1} S_{m-l, k-1}^{-1, -\alpha, -(n-j\alpha)}.
\end{equation}
\end{prop}
\begin{proof}Let $f(K_m)= \frac{V_{n+m+1, j+K_m}}{V_{n+m, j+K_m}}$ then, by definition of $W_{l,m}^{(n)}$,
$$
\mathbb{E}_{(S_1,\dots, S_{K_m}, K_m| K_n=j)} \left( \frac{V_{n+m+1, j+K_m}}{V_{n+m, j +K_m}} (l- \alpha)W_{l, m}^{(n)}\right)= 
$$
$$
= (l- \alpha) \sum_{k=1}^ {m-l+1} \mathbb{E}_{(S_1,\dots, S_{K_m}| K_m=k, K_n=j)} \left(f(k)  \sum_{i=1}^{k} 1 \{S_{i}=l|K_n=j\}\right) \times$$
$$
\times\mathbb{P}(K_m=k| K_n=j)=
$$
\begin{equation}
\label{quasi}
= (l- \alpha) \sum_{k-1=0}^ {m-l} f(k)  k \mathbb{P} (S_i=l| K_m=k, K_n=j) \mathbb{P}(K_m=k|K_n=j).
\end{equation}
Now, specializing (\ref{margnew}),
$$
\mathbb{P}(S_1=l, \dots, S_r=l, K_m=k|K_n=j)=
$$
$$= \frac{m!}{m- rl!} \frac{(k-r)!}{(l!)^r} \frac{[(1-\alpha)_{l-1}]^r}{k!} \frac{V_{n+m, j+k}}{V_{n,j}} S_{m-rl, k-r}^{-1, -\alpha, -(n-j\alpha)}
$$
and inserting the marginal for $r=1$ in (\ref{quasi}), the result follows.
\end{proof}
By analogous approach we provide a straightforward derivation for the Bayesian nonparametric estimator for the probability to observe a species represented $l$ times among  the {\it old} species, namely
$$
P_{old, l}^{m+n+1}(\alpha, V)=\frac{V_{n+m+1, j+K_m}}{V_{n+m, j+K_m}}(l-\alpha) O_{l,m}^{(n)}.
$$
\begin{prop}
Under a general $(\alpha, V)$ Gibbs partition model, for $O_{l,m}^{(n)}= \sum_{i=1}^{j} 1\{n_i+M_{i,m}=l|n_1, \dots, n_j\}$ then a Bayesian nonparametric estimator under quadratic loss function of $P_{old, l}^{m+n+1}(\alpha, V_{n,j})$ is given by
\begin{equation}
\label{discoold}
\mathbb{E}_{(M_{1,m}, \dots, M_{j,m}, K_m|n_1, \dots, n_j)} \left( \frac{V_{n+m+1, j+K_m}}{V_{n,m, j+K_m}}(l-\alpha) O_{l,m}^{(n)}\right)=
\end{equation}
$$
=(l-\alpha) \sum_{\xi=1}^{l} m_\xi {m \choose {l-\xi}} (\xi-\alpha)_{l-\xi} \sum_{k=0}^{m-l+\xi} \frac{V_{n+m+1, j+k}}{V_{n,j}} S_{m-l+\xi, k}^{-1, -\alpha, -(n-j\alpha + \xi -\alpha)}
$$
\end{prop}
\begin{proof}
Let $f(K_m)=\frac{V_{n+m+1, j+K_m}}{V_{n,m, j+K_m}}$, then
$$
\mathbb{E}_{(M_{1,m}, \dots, M_{j,m}, K_m|n_1, \dots, n_j)} \left( \frac{V_{n+m+1, j+K_m}}{V_{n,m, j+K_m}}(l-\alpha) O_{l,m}^{(n)}\right)=
$$
$$
= (l-\alpha) \sum_{k=0}^{m}f(k) \times
$$
$$\times \mathbb{E}_{(M_{1,m}, \dots, M_{j,m}| K_m=k, n_1,\dots, n_j)} \left(\sum_{i=1}^{j} 1 \left\{n_i+M_{i,m}=l | K_m=k, n_1, \dots, n_j\right\} \right)\times
$$
$$
\times \mathbb{P}(K_m=k|K_n=j),
$$
which is equal to
$$
=(l-\alpha) \sum_{k=0}^m f(k) \sum_{i:n_i \leq l} \mathbb{P} (M_{i,m}= l-n_i, K_m=k |n_1,\dots, n_j) 
$$
and by (\ref{onemargold})
$$
=(l-\alpha) \sum_{i:n_i \leq l}   {m \choose {l-n_i}} (n_i-\alpha)_{l-n_i} \sum_{k=0}^{m-l+n_i} \frac{V_{n+m+1, j+k}}{V_{n,j}} S_{m-l+n_i, k}^{-1, -\alpha, -(n-j\alpha -n_i +\alpha)}
$$
and the result follows.
\end{proof}


\appendix
\section{}
This Appendix contains some basic facts on rising and falling factorial numbers, partitions and compositions of the natural integers, together with known results and definitions of generalized {\it central} and {\it non central} Stirling numbers that are exploited in the proofs and derivations all over the paper. The main reference is \cite{pit06}. Additionally, to facilitate the reading of the results contained in \cite{lmp07, lpw08, flp12a} and \cite{flp12b}, the relationship between central and non central generalized {\it factorial} coefficients and generalized {\it Stirling} numbers is reported. \\

\subsection{Generalized rising factorials} For $n=0,1,2,\dots,$ and arbitrary real $x$ and $h$,  $(x)_{n\uparrow h}$ denotes the $n$th factorial power of $x$ with increment $h$ (also called generalized {\it rising} factorial)
\begin{equation}
\label{factdef}
(x)_{n \uparrow h}:= x(x+h)\cdots(x+(n-1)h)=\prod_{i=0}^{n-1}(x+ih)=h^n(x/h)_{n},
\end{equation}
where $(x)_{n}$ stands for $(x)_{n\uparrow 1}$, and $(x)_{h\uparrow 0}=x^h$, for which the following multiplicative law holds 
\begin{equation}
\label{multiplicative}
(x)_{n+r \uparrow h}=(x)_{n\uparrow h} (x +n h)_{r \uparrow h}.
\end{equation}
From e.g. \cite{nor04} (cfr. eq. 2.41 and 2.45) a binomial formula also holds, namely
\begin{equation}
\label{bino}
(x+y)_{n \uparrow h}=\sum_{k=0}^n {n \choose k} (x)_{k \uparrow h} (y)_{n-k \uparrow h},
\end{equation}
as well as a generalized version of the multinomial theorem, i.e.
\begin{equation}
\label{multi}
(\sum_{j=1}^p z_j)_{n \uparrow h}= \sum_{n_j \geq 0, \sum n_j=n} \frac{n!}{n_1!\cdots n_p!} \prod_{j=1}^p (z_j)_{n_j \uparrow h}.
\end{equation}
For $m_j>0$, for every $j$, and $\sum_j m_j=m$, an application of (\ref{multiplicative}) yields
\begin{equation}
\label{miaaa}
(z_j)_{n_j +m_j -1}=(z_j)_{m_j-1} (z_j +m_j -1)_{n_j}
\end{equation}
and by (\ref{multi})
$$
\sum_{n_j \geq 0, \sum n_j=n} \frac{n!}{n_1!\cdots n_p!} \prod_{j=1}^p (z_j)_{n_j +m_j -1}=\prod_{j=1}^p (z_j)_{m_j -1 }(\sum_{j=1}^p (z_j +m_j -1))_{n}=
$$
$$
=\prod_{j=1}^p (z_j)_{m_j -1}(m +\sum_{j=1}^p z_j -p)_{n}, 
$$
which makes unnecessary the proof of Lemma 1 in \cite{lpw08}.
\subsection{Partitions and compositions} A {\it partition} of the finite set $[n]=(1,\dots, n)$ into $k$ blocks is an {\it unordered} collection of non-empty disjoint sets $\{A_1,\dots, A_k\}$ whose union is $[n]$, where  the blocks $A_i$ are assumed to be listed in order of appearance, i.e. in the order of their least elements. The sequence $(|A_1|,\dots, |A_k|)$ of the sizes of blocks, $(n_1, \dots, n_k)$, defines a {\it composition} of $n$, i.e. a sequence of positive integers with sum $n$ and  $\mathcal{P}_{[n]}^k$ denotes the space of all partitions of $[n]$ with $k$ blocks. From \cite{pit06} (cfr. eq. (1.9)) the number of ways to partition $[n]$ into $k$ blocks and assign each block a $W$ combinatorial structure such that the number of $W$-structures on a set of $j$ elements is $w_j$, in terms of sum over {\it compositions} of $n$ into $k$ parts is given by
\begin{equation}
\label{bell}
B_{n,k}(w_\bullet)=\frac{n!}{k!} \sum_{(n_1,\dots, n_k)} \prod_{i=1}^k \frac{w_{n_i}}{n_i!},
\end{equation}
where $B_{n, k}(w_\bullet)$ is a polynomial in variables $w_1, \dots, w_{n-k+1}$ known as the $(n,k)$th {\it partial Bell polynomial}. 
\subsection{Generalized Stirling numbers} (For a comprehensive treatment see \cite{hsushi98}, see also \cite{pit06} Ex. 1.2.7). For arbitrary distinct reals $\eta$ and $\beta$, these are the connection coefficients $S_{n,k}^{\eta, \beta}$ defined by
\begin{equation}
\label{connection}
(x)_{n \downarrow \eta}= \sum_{k=0}^n S_{n,k}^{\eta, \beta} (x)_{k \downarrow \beta} 
\end{equation}
and correspond to
$$
S_{n,k}^{\eta, \beta}=B_{n,k}((\beta -\eta)_{\bullet -1 \downarrow \eta}),
$$
where $(x)_{n \downarrow h}$ are generalized {\it falling} factorials and $(x)_{n \downarrow -h}=(x)_{n \uparrow h}$, while $(x)_{n\downarrow 1}=(x)_{[n]}$. Hence for $\eta=-1$, $\beta=-\alpha$, and $\alpha \in (-\infty, 1)$, $S_{n,k} ^{-1, -\alpha}$ is defined by
\begin{equation}
\label{unoalpha}
(x)_{n}=\sum_{k=0}^{n} S_{n,k} ^{-1, -\alpha} (x)_{k \uparrow \alpha},
\end{equation}
and for $w_{n_i}=(1 -\alpha)_{n_i-1}$ and $\alpha \in [0,1)$, equation (\ref{bell}) yields
\begin{equation}
\label{bellalpha}
B_{n,k}((1-\alpha)_{\bullet-1})=\sum_{\{A_1,\dots, A_k\}\in \mathcal{P}_{[n]}^k }\prod_{i=1}^k (1-\alpha)_{n_i-1}=
$$
$$=\frac{n!}{k!}\sum_{(n_1,\dots,n_k)}\prod_{i=1}^k \frac{(1-\alpha)_{n_i-1}}{n_i!}=S_{n,k}^{-1,-\alpha}.
\end{equation}
In \cite{lmp07, lpw08, flp12a, flp12b} the treatment is in term of {\it generalized factorial coefficients}, which are the connection coefficients  $\mathcal{C}^\alpha_{n,k}$ defined by
\begin{equation}
\label{chara}
(\alpha y)_{n}=\sum_{k=0}^{n} \mathcal{C}^\alpha_{n,k}(y)_{k},
\end{equation}
(cfr. \cite{cha05}). 
From (\ref{factdef}) and (\ref{unoalpha}), 
if $x=y \alpha$ then
$$
(y \alpha)_{n}= \sum_{k=0}^{n} S_{n, k}^{-1, -\alpha}(y \alpha)_{k \uparrow \alpha}=\sum_{k=0}^n S_{n,k}^{-1, -\alpha} \alpha^k
(y)_{k},
$$
hence 
\begin{equation}
\label{coeff}
S_{n,k}^{-1, -\alpha}=\frac{\mathcal{C}_{n,k}^\alpha}{\alpha^k}.
\end{equation}
The representation (37) in \cite{lpw08}, (\cite{tos39}), also holds for generalized Stirling numbers with the obvious changes (cfr. e.g. \cite{pit06}, Eq. 3.19). Additionally, specializing formula (16) in \cite{hsushi98}, the following convolution relation holds, which defines {\it non-central} generalized Stirling numbers
\begin{equation}
\label{convo}
S_{n,k}^{-1, -\alpha, \gamma}= \sum_{s=k}^{n} {n \choose s} S_{s,k}^{-1, -\alpha} (-\gamma)_{n-s},
\end{equation}
and by (\ref{coeff}), 
\begin{equation}
\label{coeffsti}
\mathcal{C}_{n,k}^{\alpha, \gamma}=\alpha^k S_{n,k}^{-1, -\alpha, \gamma}= \sum_{s=k}^{n} {n \choose s} \mathcal{C}_{s,k}^{\alpha} (-\gamma)_{n-s}.
\end{equation} 
Hence the following variation of equation (38) in \cite{lpw08} defines {\it non-central} generalized Stirling numbers as the connection coefficients $S_{n, k}^{-1, -\alpha, \gamma}$ such that
\begin{equation}
\label{noncentralsti}
(y \alpha-\gamma)_{n}= \sum_{k=0}^{n} S_{n, k}^{-1, -\alpha, \gamma} \alpha^k
(y)_{k}=\sum_{k=0}^{n} S_{n, k}^{-1, -\alpha, \gamma}
(y\alpha)_{k \uparrow \alpha} .
\end{equation}

\subsection{Factorial moments and discrete distributions} (Cfr. e.g. \cite{johkot05}).  Falling factorial moments of a discrete r.v. $X$ provide the distribution function by the relationship 
\begin{equation}
\label{momprob}
\mathbb{P}(X=x)= \sum_{r \geq 0} \frac{(-1)^r}{x! r!} \mathbb{E}[(X)_{[x+r]}]
\end{equation}
and standard moments by the definition as connection coefficients of the Stirling numbers of the second kind $S_{r,j}^{0,1}$, (cfr.  Eq. (\ref{connection}))
\begin{equation}
\label{momfact}
\mathbb{E}[(X)^r] = \sum_{j=0}^r S_{r,j}^{0,1} \mathbb{E}[(X)_{[r]}].
\end{equation}


\end{document}